\documentclass[11pt]{amsart}
\usepackage{amsmath, amsthm, amscd, amssymb,amsfonts}
 \usepackage{color}
 \usepackage{mathrsfs}
 \usepackage[all]{xy}
 \usepackage{hyperref}

\newtheorem{thm}{Theorem}[section]
\newtheorem*{thmA}{Theorem A}
\newtheorem*{thmB}{Theorem B}
\newtheorem*{thmC}{Theorem C}
\newtheorem*{thmD}{Theorem D}
\newtheorem{cor}[thm]{Corollary}
\newtheorem*{cor*}{Corollary}
\newtheorem{lem}[thm]{Lemma}
\newtheorem{prop}[thm]{Proposition}

\newtheorem*{con*}{Conjecture}
\newtheorem*{prob*}{Problem}

\theoremstyle{definition}
\newtheorem{defn}[thm]{Definition}

\theoremstyle{remark}
\newtheorem{rem}[thm]{Remark}

\begin{document}
\title{Revisiting Farrell's nonfiniteness of Nil}

\author[Jean-Fran\c{c}ois Lafont]{Jean-Fran\c{c}ois Lafont$^\dagger$}
\address{Department of Mathematics
The Ohio State University,
Columbus, Ohio, 43210 U.S.A.}
\email{jlafont@math.ohio-state.edu}

\author[Stratos Prassidis]{Stratos Prassidis$^*$}
\address{Department of Mathematics,
University of the Aegean,
Karlovassi, Samos, 83200 Greece}
\email{prasside@aegean.gr}

\author[Kun Wang]{Kun Wang$^\star$}
\address{Department of Mathematics
The Ohio State University,
Columbus, Ohio, 43210 U.S.A.}
\email{kwang@math.ohio-state.edu}

\thanks{$^\dagger$ The research of the first author is partially supported by the NSF, under grant DMS-1207782}
\thanks{$^*$The research of the second author has been co-financed by the European Union (European Social Fund - ESF) and Greek national funds through the Operational Program ``Education and Lifelong Learning" of the National Strategic Reference Framework (NSRF) - Research Funding Program: THALIS}
\thanks{$^\star$The research of the third author is partially supported by an Ohio State University Presidential 
Fellowship award.}


\begin{abstract}
We study Farrell Nil-groups associated to a finite order automorphism of a ring $R$. 
We show that any such Farrell Nil-group is either trivial, or infinitely generated (as an abelian
group). Building on this first result, we then show that any finite group that occurs in such a Farrell Nil-group 
occurs with infinite multiplicity. 
If the original finite group is a direct summand, then the countably infinite sum of the finite subgroup also appears 
as a direct summand. We use this to deduce a structure theorem for countable Farrell Nil-groups with finite exponent.
Finally, as an application, we show that if $V$ is any virtually cyclic group, 
then the associated Farrell or Waldhausen Nil-groups can always be expressed as a countably infinite sum of 
copies of a finite group, provided they have finite exponent (which is always the case in dimension $0$).
\end{abstract}

\maketitle

\section{Introduction}

For a ring $R$ and an automorphism $\alpha: R\rightarrow R$, one can form the twisted polynomial ring
$R_\alpha[t]$, which as an additive group coincides with the polynomial ring $R[t]$, but with product given by
$(rt^i)(st^j) = r\alpha^{-i}(s)t^{i+j}$. 
There is a natural augmentation map $\varepsilon : R_\alpha[t] \rightarrow R$ induced by setting 
$\varepsilon (t) =0$. For $i\in \mathbb Z$, the {\it Farrell twisted Nil-groups}  $NK_i(R, \alpha) := \ker ({\varepsilon}_*)$
are defined to be the kernels of the induced $K$-theory map $\varepsilon_*: K_i(R_\alpha[t]) \rightarrow K_i(R)$. 
This induced map is split injective, hence $NK_i(R, \alpha)$ can be viewed as a 
direct summand in $K_i(R_\alpha[t])$. In the special case where the automorphism $\alpha$ is the identity,
the ring $R_\alpha[t]$ is just the ordinary polynomial ring $R[t]$, and the Farrell twisted Nil reduces to the
ordinary {\it Bass Nil-groups}, which we just denote by $NK_i(R)$. We establish the following:

\begin{thmA}\label{main} Let $R$ be a ring, $\alpha:R\rightarrow R$ a ring automorphism of finite order, and
$i \in \mathbb Z$. Then $NK_i(R, \alpha)$ is either trivial, or infinitely generated as an abelian group.
\end{thmA}

\vskip 10pt

The proof of this result relies heavily on a method developed by Farrell \cite{fa}, who first showed in 1977 
that the lower 
Bass Nil-groups $NK_*(R)$ with $*\leq 1$, are always either trivial, or infinitely generated. This result was subsequently 
extended to the higher Bass Nil-groups 
$NK_*(R)$ with $*\geq 1$ by Prasolov \cite{pr} (see also van der Kallen \cite{ka}). For Farrell's twisted Nils, when the
automorphism $\alpha$ has finite order, Grunewald \cite{gr-agt} and Ramos \cite{ra} independently established 
the corresponding result for $NK_*(R, \alpha)$ when $*\leq 1$. All these papers used the same basic idea, which
we call {\it Farrell's Lemma}. We exploit the same idea, and establish our own version of Farrell's Lemma (and
prove the theorem) in Section \ref{non-finiteness}. 

\begin{rem} Farrell's original proof of his lemma used the transfer map on $K$-theory. Na\"ively, one might want to 
try to prove {\bf Theorem A} as follows: choose n so that $\alpha^n=\alpha$. Then
there is a ring homomorphism from $A=R_\alpha[t]$
to $B=R_\alpha[s]$ sending $t \mapsto s^n$. Call the
induced map on $K$-theory $F_n:K(A) \rightarrow K(B)$. Since $B$ is a free (left)
$A$-module of rank $n$, the transfer map $V_n$ is
also defined, and $G_n:=V_n\circ F_n=\mu _n$ (multiplication by $n$). 
Then follow Farrell's original 1977 argument verbatim to conclude the proof. Unfortunately this approach does not work, 
for two reasons. 

Firstly, the identity $G_n =\mu_n$ does not hold in the
twisted case (basically due to the fact that $\oplus _nA$ and $B$ are not isomorphic {\it as bimodules}). We do
not explicitly know what the map $G_n$ does on $K$-theory, but it is definitely {\bf not} multiplication by an integer.
Instead, we have the somewhat more complicated identity given in part (2) of our Lemma \ref{farrell-lemma}, but which
is still sufficient to establish the Theorem. 

Secondly, while it is possible to derive the identity in part (2) of Lemma \ref{farrell-lemma} using the transfer map as in 
Farrell's original argument, it is not at all clear how to obtain the analogue of part (3) in higher dimensions by working at
the level of 
K-theory groups. Instead, we have to work at the level of categories, specifically, with the \textit{Nil-category} 
$NIL(R;\alpha)$ (see Section \ref{nilcat}), in order to ensure property (3). The details are in \cite{gr-topology}.

\end{rem}


Next we refine somewhat the information we have on these Farrell Nils, by focusing on the finite subgroups
arising as direct summands. In section \ref{fin-sgps}, we establish:

\begin{thmB}\label{main} Let $R$ be a ring, $\alpha:R\rightarrow R$ a ring automorphism of finite order, and
$i \in \mathbb Z$. If $H\leq NK_i(R, \alpha)$ is a finite subgroup, then $\bigoplus_{\infty} H$ also appears
as a subgroup of  $NK_i(R,\alpha)$. Moreover, if $H$ is a direct summand in $NK_i(R, \alpha)$, then so is 
$\bigoplus_{\infty} H$. 
\end{thmB}

In the statement above, and throughout the paper, $\oplus_{\infty} H$ denotes the direct sum of countably infinitely
many copies of the group $H$.
\textbf{Theorem B} together with some group theoretic facts enable 
us to deduce a structure theorem for certain Farrell Nil-groups. In section \ref{Nils-VC}, we prove:

\begin{thmC} Let $R$ be a countable ring, $\alpha: R\rightarrow R$ a ring automorphism of finite order, and $i\in\mathbb Z$. If $NK_i(R,\alpha)$ has finite exponent, then there exists a finite abelian group $H$, so that
$NK_i(R,\alpha)\cong\bigoplus_{\infty} H$.
\end{thmC}

A straightforward corollary of \textbf{Theorem C} is the following:

\begin{cor} \label{nk0}
Let $G$ be a finite group, $\alpha\in Aut(G)$. Then there exists a finite abelian group $H$, whose exponent divides some power of $|G|$, with the property that $NK_0(\mathbb Z G,\alpha)\cong\bigoplus_{\infty}H$.
\end{cor}
\begin{proof} F. Connolly and S. Prassidis in \cite{cp} proved that $NK_0(\mathbb Z G, \alpha)$ has finite exponent when $G$ is finite. A. KuKu and G. Tang \cite[Theorem 2.2]{kt} showed that $NK_i(\mathbb Z G,\alpha)$
is $|G|$-primary torsion for all $i\ge 0$. These facts together with \textbf{Theorem C} above complete the proof.
\end{proof}

\begin{rem} It is a natural question whether the above Corollary holds in dimensions other than zero.  In negative 
dimensions $i<0$, Farrell and Jones showed in \cite{fj2} that $NK_i(\mathbb Z G,\alpha)$ always vanishes when $G$
is finite. In positive dimensions $i>0$, there are partial results. As mentioned in the proof above,  Kuku and Tang  
\cite[Theorem 2.2]{kt} showed that $NK_i(\mathbb Z G,\alpha)$
is $|G|$-primary torsion.  Grunewald  \cite[Theorem 5.9]{gr-topology} then generalized their result to 
\textit{polycyclic-by-finite} groups in all dimensions. He showed that, for all $i\in\mathbb Z$,  
$NK_i(\mathbb Z G,\alpha)$ is $mn$-primary
torsion for every polycyclic-by-finite group $G$ and every group automorphism $\alpha: G\rightarrow G$  of finite order, 
where $n=|\alpha|$ and $m$ is the index of some poly-infinite cyclic subgroup of $G$ (such a subgroup always exists). 
However, although we have these nice results on the possible orders of torsion elements, it seems there are
no known results on the exponent of these Nil-groups. This is clearly a topic for future research. 
\end{rem}

\begin{rem}
As an example in dimension greater than zero, Weibel \cite{we} showed that $NK_1(\mathbb ZD_4) \neq 0$, where $D_4$ denotes the 
dihedral group of
order $8$. He also constructs a surjection $\bigoplus_{\infty} (\mathbb Z_2\oplus \mathbb Z_4) 
\rightarrow NK_1(\mathbb ZD_4)$, showing that this group has exponent $2$ or $4$. It follows from our Corollary
that the group $NK_1(\mathbb ZD_4)$ is isomorphic to one of the three groups $\bigoplus_{\infty} (\mathbb Z_2\oplus \mathbb Z_4)$, $\bigoplus_{\infty} \mathbb Z_4$, or $\bigoplus_{\infty} \mathbb Z_2$.
\end{rem}

For our next application, we recall that there is, for any group $G$, an assembly
map $H_n^G(\underbar EG; \textbf K_{\mathbb Z}) \rightarrow K_n(\mathbb Z[G])$, where $H_*^?(-; \textbf K_{\mathbb Z})$ denotes the specific 
equivariant generalized homology theory appearing in the $K$-theoretic Farrell-Jones isomorphism conjecture with coefficient in $\mathbb Z$,
and $\underbar EG$ is a model for the classifying space for proper $G$-actions. We refer the reader to 
Section \ref{Nils-VC} for a discussion of these notions, as well as for the proof of:

\begin{thmD}\label{main} For any  virtually cyclic group $V$, there exists a finite abelian
group $H$ with the property that there is an isomorphism:
$$\bigoplus _{\infty} H \cong \text{CoKer}\Big( H_0^V(\underbar EV; \textbf K_{\mathbb Z}) \rightarrow K_0(\mathbb Z[V])\Big)$$
The same result holds in dimension $n$ whenever $\text{CoKer}\Big( H_n^V(\underbar EV; \textbf K_{\mathbb Z}) \rightarrow K_n(\mathbb Z[V])\Big)$ has finite exponent.
\end{thmD}
We conclude the paper with some general remarks and open questions in Section \ref{concluding}.


\section{Some exact functors}\label{nilcat}

In this section, we define various functors that will be used in our proofs.
Let $R$ be an associative ring with unit and $\alpha: R\rightarrow R$ be a ring automorphism of finite order, 
say $|\alpha| = n$.  For each integer $i\in\mathbb Z$, denote by $R_{\alpha^i}$ the $R$-bimodule which coincides 
with $R$ as an abelian group, but with bimodule structure given by 
$r\cdot x:=rx$ and $x\cdot r:=x\alpha^i(r)$ (where $x\in R_{\alpha^i}$ and $r\in R$).  Note that as left (or as right) 
$R$-modules, $R_{\alpha^i}$ and $R$ are 
isomorphic, but they are in general {\it not} isomorphic  as $R$-bimodules. For each right $R$-module $M$ and integer 
$i$, define a new right $R$-module $M_{\alpha^i}$ as follows: as abelian groups, $M_{\alpha^i}$ is the same as $M$, 
however the right $R$-module structure on $M_{\alpha^i}$ is given by $x\cdot r:=x\alpha^i(r)$. Clearly 
$M_{\alpha^{n}}=M$ and $(M_{\alpha^i})_{\alpha^j}=M_{\alpha^{i+j}}$ as right $R$-modules. We could have defined 
$M_{\alpha^i}=M\otimes_RR_{\alpha^i}$, however this has the slight disadvantage that the above equalities 
would not hold -- we would only have natural isomorphisms between the corresponding functors.

\vskip 10pt

Let $\textbf{P}(R)$ denote the category of finitely generated right projective $R$-modules. For each $i\in\mathbb Z$, 
there is an exact functor $S_i: \textbf{P}(R)\rightarrow\textbf{P}(R)$ given by $S_i(P):=P_{\alpha^i}$ on objects and 
$S_i(\phi)=\phi$ on morphisms. Note that if we forget about the right $R$-module structures, and just view these
as abelian groups and group homomorphisms, then each $S_i$ is just the 
identity functor.  Clearly $S_i\circ S_j=S_j\circ S_i=S_{i+j}$ and $S_n=Id$, so the map $i\mapsto S_i$ defines a 
functorial $\mathbb Z$-action on the category $\textbf{P}(R)$, which factors through a functorial $\mathbb Z_n$-action
(recall that $n$ is the order of the ring automorphism $\alpha$). 

\vskip 10pt

We are interested in the {\it Nil-category} $NIL(R; \alpha)$. Recall that objects of this category are of the form $(P, f)$, 
where $P$ is an object in $\textbf{P}(R)$ and $f: P\rightarrow P_{\alpha}=S_1(P)$ is a right $R$-module homomorphism
which is {\it nilpotent}, in the sense that a high enough composite map of the following form is the zero map:
$$
\xymatrix{
P \ar[rrrrr]^-{S_{k-1}(f)\circ S_{k-2}(f)\circ\cdots\circ S_1(f)\circ f} & & & & & P_{\alpha^{k}} \\
}$$
A morphism $\phi: (P,f)\rightarrow(Q,g)$ in $NIL(R;\alpha)$ is given by a morphism 
$\phi: P\rightarrow Q$ in $\textbf{P}(R)$ which makes the obvious diagram commutative, 
i.e. $S_1(\phi)\circ f = g\circ \phi$.   We have two exact functors 
$$F: NIL(R; \alpha)\rightarrow\textbf{P}(R),\ \  F(P,f)=P$$
$$G: \textbf{P}(R)\rightarrow NIL(R; \alpha),\ \ G(P)=(P, 0)$$
which give rise to a splitting of the $K$-theory groups $K_i(NIL(R; \alpha))=K_i(R)\oplus Nil_i(R;\alpha)$, 
where $Nil_i(R;\alpha):=\text{Ker}\big(K_i(NIL(R;\alpha))\rightarrow K_i(R)\big), i\in\mathbb N$.

\begin{rem}
The Farrell Nil-groups $NK_*(R, \alpha)$ mentioned in the introduction coincide, with a dimension shift, with the
groups $Nil_*(R;\alpha^{-1})$ defined above. More precisely, one has for every $i\geq 1$ an isomorphism 
$NK_i(R, \alpha) \cong Nil_{i-1}(R;\alpha^{-1})$ (\cite[Theorem 2.1]{Gr2}). 
\end{rem}

We now introduce two exact functors on the category $NIL(R;\alpha)$ which will play an important role in our proofs. 
On the level of $K$-theory, one of these yields the twisted analogue of the Verscheibung operators, while the other 
gives the classical Frobenius operators.

\begin{defn}[Twisted Verscheibung functors] \label{Verscheibung}
For each positive integer $m$, define the twisted Verscheibung functors 
$V_m : NIL(R;\alpha)\rightarrow NIL(R;\alpha)$ as follows.
On objects, we set
$$V_m\big((P, f)\big)=(P\oplus P_{\alpha^{-1}}\oplus P_{\alpha^{-2}}\oplus\cdots\oplus P_{\alpha^{-mn}}, \overline{f})
=\Big(\displaystyle\sum_{i=0}^{mn} P_{\alpha^{-i}}, \overline{f}\Big)=\Big(\displaystyle\sum_{i=0}^{mn}S_{-i}(P), 
\overline{f}\Big)$$
where the morphism 
$$\overline{f}: \displaystyle\sum_{i=0}^{mn} P_{\alpha^{-i}}\longrightarrow \Big(\displaystyle\sum_{j=0}^{mn} 
P_{\alpha^{-j}}\Big)_{\alpha}=\displaystyle\sum_{j=0}^{mn} P_{\alpha^{-j+1}} $$
is defined component-wise by the maps $f_{ij}: P_{\alpha^{-i}}\rightarrow P_{\alpha^{-j+1}}$ given via the formula
$$f_{ij}=
\begin{cases}
id & \text{if}\ j=i+1, 0\le i\le mn-1\\
f & \text{if}\ i=mn, j=0\\
0 & \text{otherwise}
\end{cases}
$$
In the proof of Lemma \ref{relation} below, we will see that $\overline{f}$ is nilpotent, so that $V_m\big((P, f)\big)$ 
does indeed define an object
in the category $NIL(R;\alpha)$.  If $\phi:(P,f)\rightarrow (Q, g)$ is a morphism in the category $NIL(R;\alpha)$, we define the morphism 
$$V_m(\phi):\Big(\displaystyle\sum_{i=0}^{mn} P_{\alpha^{-i}}, \overline{f}\Big)\rightarrow \Big(\displaystyle\sum_{i=0}^{mn} Q_{\alpha^{-i}}, \overline{g}\Big)$$ 
via the formula $V_m(\phi)=\displaystyle\sum_{i=0}^{mn}S_{-i}(\phi)$. One checks that 
(i) $\overline{g}\circ V_m(\phi)=S_1(V_m(\phi))\circ\overline{f}$, (ii) $V_m(id)=id$ and (iii) $V_m(\phi\circ\psi)=V_m(\phi)\circ V_m(\psi)$, so that $V_m$ is indeed a functor.  Moreover, $V_m$ is exact because each $S_{-i}$ is exact. 
\end{defn}

\begin{defn}[Frobenius functors]
For each positive integer $m$, define the Frobenius functors $F_m : NIL(R;\alpha)\rightarrow NIL(R;\alpha)$ as follows. 
On objects, we set $F_m\big((P,f)\big)=(P, \tilde{f})$ where $\tilde f$ is the morphism defined by the composition
$$
\xymatrix{
P \ar[rrrrrr]^-{S_{mn}(f)\circ S_{mn-1}(f)\circ\cdots\circ S_1(f)\circ f} & & & & & & P_{\alpha^{mn+1}}=P_{\alpha} \\
}$$
(recall that the ring automorphism $\alpha$ has order $|\alpha|=n$). It is immediate that the map $\tilde f$ is 
nilpotent, so that $F_m\big((P,f)\big)$ is indeed an object in $NIL(R;\alpha)$.
Now if $\phi: (P,f)\rightarrow (Q, g)$ is a morphism in the category $NIL(R;\alpha)$, we define the morphism 
$F_m(\phi): (P, \tilde f)\rightarrow (Q, \tilde g)$ to coincide with the morphism $\phi$.
It is obvious that $F_m(id)=id$ and $F_m(\phi\circ\psi)=F_m(\phi)\circ F_m(\psi)$, and one can easily check that
$\tilde g\circ\phi=S_1(\phi)\circ\tilde f$, so that $F_m$ is a genuine functor. Clearly $F_m$ is exact. 
\end{defn}

\begin{defn}[$\alpha$-twisting functors] 
For each $i\in\mathbb Z$, we define the exact functor $T_i: NIL(R;\alpha)\rightarrow NIL(R;\alpha)$ as follows. On objects, we set $T_i\big((P,f)\big)=\big(S_{-i}(P), S_{-i}(f)\big)$, and if
$\phi:(P,f)\rightarrow(Q,g)$ is a morphism, we set $T_i(\phi)$ to be the morphism $S_{-i}(\phi): S_{-i}(P) \rightarrow 
S_{-i}(Q)$. Observe that, as with the functors $S_i$ on the category $\textbf{P}(R)$, the functors $T_i$ define a 
functorial $\mathbb Z$-action on the category $NIL(R;\alpha)$, which factors through a functorial $\mathbb Z_n$-action.
\end{defn}

The relationship between these various functors is described in the following Lemma. We will write $G_m$ for the 
composite exact functor $G_m=F_m\circ V_m$.

\begin{lem}\label{relation}
We have the equality $G_m=\displaystyle\sum_{i=0}^{mn}T_i$.
\end{lem}

\begin{proof}  Let $(P, f)$ be an object in $NIL(R;\alpha)$. Then we have 
$G_m\big((P, f)\big)=\Big(\displaystyle\sum_{i=0}^{mn}S_{-i}(P), \tilde{\overline{f}}\Big)$, where 
$\tilde{\overline{f}}=S_{mn}(\overline{f})\circ S_{mn-1}(\overline{f})\circ\cdots\circ S_1(\overline{f})\circ\overline{f}$.  
Note that if we forget the right $R$-module structures, each $S_i$ is the identity functor on abelian groups.  So as a 
morphism of abelian groups, $\tilde{\overline{f}}=\overline{f}^{mn+1}$. Now recall that $\overline{f}$ is a morphism
which cyclicly permutes the $mn+1$ direct summands occuring in its source and target. Using this observation, it is 
then easy to see that $\tilde{\overline{f}}=\overline{f}^{mn+1}$ is diagonal and equal to 
$\displaystyle\sum_{i=0}^{mn}S_{-i}(f)$. So on the level of objects, $G_m$ and $\displaystyle\sum_{i=0}^{mn}T_i$ 
agree. From this, we also see that $\overline f$ is nilpotent (as was indicated in Definition \ref{Verscheibung}). 
It is obvious that they agree on morphisms. 
\end{proof}

\begin{rem} It is natural to consider the more general case when $\alpha: R\rightarrow R$ has finite order in the 
{\it outer automorphism group} of the ring $R$, i.e. there exists $n\in\mathbb N$ and 
a unit $u\in R$ so that $\alpha^n(r)=uru^{-1}, \forall r\in R$. In this situation, we have for any right $R$-module $M$ 
and integer $m$, an isomorphism $\tau_{m, M}: M_{\alpha^{mn}}\rightarrow M$, $\tau_{m, M}(r):=ru^m$ of right 
$R$-modules. This gives rise to a natural isomorphism between the functors $S_{mn}$ and $S_0=Id$.  It is then easy
to similarly define twisted Verscheibung functors and Frobenius functors, and to verify
an analogue of Lemma \ref{relation}.  However, in this case, we generally do {\bf not} have that $T_n$ is naturally 
isomorphic to $T_0$, unless $\alpha$ fixes $u$. This key issue prevents our proof of Farrell's Lemma 
\ref{farrell-lemma}(2) below (which is the key to the proof of our main theorems) to go through in this more 
general setting.
\end{rem}


\section{Non-finiteness of Farrell Nils}\label{non-finiteness}

This section is devoted to proving {\bf Theorem A}.


\subsection{A version of Farrell's Lemma.}
We are now ready to establish our analogue of Farrell's key lemmas from his paper \cite{fa}.

\begin{lem}\label{farrell-lemma} The following results hold:

\begin{enumerate}

\item $\forall j\in\mathbb N$, the induced morphisms $K_j(V_m), K_j(F_m): K_j(NIL(R;\alpha))\rightarrow K_j(NIL(R;\alpha))$ on $K$-theory map the 
summand $Nil_j(R;\alpha)$ to itself;

\item $\forall j, m\in\mathbb N$, the identity $(2+mn)K_j(G_m)-K_j(G_m)^2=\mu_{1+mn}$ holds, where the map 
$\mu_{1+mn}$ is multiplication by $1+mn$;

\item $\forall j\in\mathbb N$ and each $x\in Nil_j(R;\alpha)$, there exists a positive integer $r(x)$, such that $K_j(F_m)(x)=0$ for all 
$m\ge r(x)$.

\end{enumerate}
\end{lem}

\begin{proof} (1) Let $H_m:=\displaystyle\sum_{i=0}^{mn}S_{-i}: \textbf{P}(R)\rightarrow\textbf{P}(R)$, one then easily checks $F\circ V_m=H_m\circ F$. We also have $F\circ F_m=F$. Statement (1) follows easily from these.

\vskip 5pt

\noindent (2) By the Additivity Theorem for algebraic K-theory, Lemma \ref{relation} immediately gives us that
$$K_j(G_m)=\displaystyle\sum_{i=0}^{mn}K_j(T_i)=id+m\displaystyle\sum_{i=1}^{n}K_j(T_i)$$
(recall that the functors $T_i$ are $n$-periodic). Now let us evaluate the square of the map $K_j(G_m)$:
\begin{align*}
K_j(G_m)^2 &=\Big(id+m\displaystyle\sum_{i=1}^{n}K_j(T_i)\Big)\Big(id+m\displaystyle\sum_{l=1}^{n}K_j(T_l)\Big)\\
&=id+2m\displaystyle\sum_{i=1}^{n}K_j(T_i)+m^2\displaystyle\sum_{i=1}^{n}\displaystyle\sum_{l=1}^nK_j(T_{i+l})\\
&=id+2m\displaystyle\sum_{i=1}^{n}K_j(T_i)+m^2\displaystyle\sum_{i=1}^{n}\displaystyle\sum_{l=1}^nK_j(T_{l})\\
&=id+2m\displaystyle\sum_{i=1}^{n}K_j(T_i)+m^2n\displaystyle\sum_{l=1}^nK_j(T_{l})\\
&=id+(2m+m^2n)\displaystyle\sum_{i=1}^{n}K_j(T_i)
\end{align*}
In the third equality above, we used the fact that the $T_i$ functors are $n$-periodic, so that shifting the index on
the inner sum by $i$ leaves the sum unchanged. Finally, substituting in the expression we have for $K_j(G_m)$ 
and the expression we derived for $K_j(G_m)^2$, we see that:
\begin{align*}
&(2+mn)K_j(G_m)-K_j(G_m)^2 \\
&= (2+mn)\Big(id+m\displaystyle\sum_{i=1}^{n}K_j(T_i)\Big) - \Big(id+(2m+m^2n)\displaystyle\sum_{i=1}^{n}K_j(T_i) \Big)\\
&= (2+mn)id-id =\mu_{(1+mn)}
\end{align*}
completing the proof of statement (2).

\vskip 5pt

\noindent(3) This result is due to Grunewald \cite[Proposition 4.6]{gr-topology}.
\end{proof}


\subsection{Proof of Theorem A} The proof of {\bf Theorem A} now follows easily. Let us focus on the
case where $i\geq 1$, as the case $i\leq 1$ has already been established by Grunewald \cite{gr-agt}
and Ramos \cite{ra}. 
So let us assume that the Farrell Nil-group $NK_i(R, \alpha) \cong Nil_{i-1}(R;\alpha^{-1})$ is non-trivial
and finitely generated, where $i\geq 1$. Then one can find arbitrarily large positive integers $m$ with the property 
that the map $\mu_{(1+mn)}$ is an injective map from $Nil_{i-1}(R;\alpha^{-1})$ to itself (for example, one can
take $m$ to be any multiple of the order of the torsion subgroup of $Nil_{i-1}(R;\alpha^{-1})$). From Lemma 
\ref{farrell-lemma}(2), we can factor the map $\mu_{(1+mn)}$ as a composite 
$$\mu_{(1+mn)} = \big(\mu_{(2+mn)}-K_j(G_m)\big)\circ K_j(G_m)$$
and hence there is an arbitrarily large sequence of integers $m$ with the property that the corresponding
maps $K_j(G_m) = K_j(F_m) \circ K_j(V_m)$ are injective. This implies that there are infinitely many integers
$m$ for which the map $K_j(F_m)$ is non-zero. 

On the other hand, let $x_1, \ldots, x_k$ be a finite set of generators for the abelian group $Nil_{i-1}(R;\alpha^{-1})$.
Then from Lemma \ref{farrell-lemma}(3), we have that for any $m\geq \max \{r(x_i)\}$, the map $K_j(F_m)$ is 
identically zero, a contradiction. This completes the proof of {\bf Theorem A}.


\section{Finite subgroups of Farrell Nil-groups}\label{fin-sgps}


\subsection{A Lemma on splittings.}

In order to establish {\bf Theorem B}, we will need an algebraic lemma for recognizing when two direct summands
inside an ambient group jointly form a direct summand.

\begin{lem}\label{splitting-lemma}
Let $G$ be an abelian group and $H<G$, $K<G$ be a pair of subgroups. Suppose that $H\cap K=\{0\}$, 
and that there are two retractions $\lambda: G\rightarrow H$ and $\rho: G\rightarrow K$ with the property that 
$\lambda(K)=\{0\}$. Then there exists a subgroup $L<G$, which is isomorphic to $H$, and such that $L\oplus K$ is 
also a direct summand of $G$.
\end{lem}

\begin{proof} Consider the homomorphism $Id-\rho: G\rightarrow G$.  Let $L=\{h-\rho(h)| h\in H\}$ be the image of $H$
under this homomorphism. We first note that $(Id-\rho)|_H: H\rightarrow L$ is an isomorphism. It is certainly a 
surjection. Now suppose that $h-\rho(h)=0$ for some $h\in H$. Then $h=\rho(h)\in K$,
which forces $h\in H\cap K=\{0\}$, and hence $h=0$. This shows that $Id-\rho|_H$ is also an injection. So we now
know that $H\cong L$.  Next we observe that $L\cap K=\{0\}$. To see this, take any 
$h-\rho(h)\in L\cap K$. Then since 
$\rho(h)\in K$, we must also have $h\in K$. But then $h\in H\cap K = \{0\}$, forcing $h=0$ and hence $h-\rho(h)=0$.

Now define $\tau:=(Id-\rho)\circ\lambda: G\rightarrow L$. For any $h-\rho(h)\in L$, we have
$$\tau(h-\rho(h))=(Id-\rho)(\lambda(h-\rho(h)))=(Id-\rho)(h)=h-\rho(h)$$
where the second equality holds because $\lambda(h)=h$ (since $h\in H$ and $\lambda$ is a retraction onto $H$) 
and $\lambda(\rho(h))=0$ (since $\rho(h)\in K$ and $\lambda(K)=\{0\}$ by hypothesis). This verifies that the map 
$\tau: G\rightarrow L$ is a retraction. Clearly $\tau(K)=\{0\}$ because $\lambda(K)=0$.  We finally note that 
$\rho(L)=0$, because $\rho(h-\rho(h))=\rho(h)-\rho(h)=0$. We thus have two orthogonal retractions $\tau$ and $\rho$. 
Now define $\sigma: G\rightarrow L\oplus K$ by $\sigma(x)=(\tau(x), \rho(x))$. Since $L\cap K=\{0\}$ and $\tau, \rho$ 
are orthogonal, one easily checks that $\sigma$ is a retraction. Hence $L\oplus K$ is a direct summand of $G$, which 
proves the lemma.
\end{proof}


\subsection{Proof of Theorem B} We are now ready to prove {\bf Theorem B}. In order to simplify the notation,
we will simply write $V_m$ for $K_j(V_m)$, and use a similar convention for $F_m$ and $G_m$.

\vskip 10pt

\noindent \underline{Case $i\geq 1$.}
We first consider the case when $i\geq 1$, and recall that $NK_i(R, \alpha) \cong Nil_{i-1}(R; \alpha^{-1})$. 
Let $H<Nil_{i-1}(R;\alpha^{-1})$ be a finite subgroup. According to Lemma 
\ref{farrell-lemma}(3), since $H$ is finite, there exists an integer $r(H) = \max_{x\in H}\{r(x)\}$, so that 
$F_m(H)=0$ for all $m>r(H)$.  Let 
$S\subset\mathbb N$  consist of all natural number $m>r(H)$ such that GCD($1+mn, |H|$)=1.  $S$ contains every
multiple of $|H|$ which is greater than $r(H)$, so is an infinite set. Consider the morphisms
$$\xymatrix{
Nil_{i-1}(R;\alpha^{-1})\ar[r]^-{V_m}& Nil_{i-1}(R; \alpha^{-1})\ar[r]^-{F_m}& Nil_{i-1}(R;\alpha^{-1})}
$$
so that the composite is the morphism $G_m$, and define the subgroup $H_m\leq Nil_{i-1}(R; \alpha^{-1})$ 
to be $H_m:=V_m(H)$. 
By the defining property of the set $S$, we have that for $m\in S$, $(\mu_{2+mn}-G_m)\circ G_m=\mu_{1+mn}$ is an 
isomorphism when restricted to $H$. Hence $G_m$ is a monomorphism when restricted to $H$, forcing $V_m$ to
also be a monomorphism when restricted to $H$. So for all $m\in S$, we see that  $H_m\cong H$.

We now claim that there is an $m\in S$, so that $H_m\cap H=\{0\}$. Assume not. Then for all $m\in S$, 
$H_m\cap H\ne\{0\}$. Since $H$ contains only finitely many non-zero elements, and $S$ is an infinite set, there is a 
non-zero $x\in H$ and an infinite subset $S'\subset S$, such that $x\in H_m\cap H$ holds for all $m\in S'$.  For each 
$m\in S'$, there is $y_m\in H$ so that $V_m(y_m)=x$.  Again, $H$ is finite, so we can find a single non-zero $y\in H$ 
and an infinite subset $S''\subset S$ with the property that for all $m\in S''$, we have $V_m(y)=x$. Applying 
$(\mu_{2+mn}-G_m)F_m$ to this equation, we obtain an infinite set of indices $m$ with the property that 
$$((\mu_{2+mn}-G_m)F_mV_m)(y)=(\mu_{2+mn}-G_m)(F_m(x))$$
Therefore 
$$(\mu_{2+mn}-G_m)(F_m(x))=(1+mn)y$$
The right hand side of the equation is non-zero for all $m\in S''$, since $y\in H$ and GCD($1+mn, |H|$)=1. 
But since  $F_m(H)=0$ for all $m\in S$, the left hand 
side vanishes, which gives us a contradiction. We conclude that there must be an $m$ so that $H_m\cap H=\{0\}$ and 
$H_m\cong H$. Hence $H\oplus H<Nil_{i-1}(R;\alpha^{-1})$. Applying the same argument to $H\oplus H$ and so on, 
we conclude $\oplus_{\infty}H<Nil_{i-1}(R;\alpha^{-1})$. 

Next we claim that, if the original subgroup $H$ was a direct summand in $Nil_{i-1}(R;\alpha^{-1})$, then we can find a copy of 
$H\oplus H$ is also a direct summand in $Nil_{i-1}(R;\alpha^{-1})$, and which {\it extends} the original direct summand (i.e. the first copy of $H$
inside the direct summand $H\oplus H$ coincides with the original $H$).

To see this, let us assume $H<Nil_{i-1}(R;\alpha^{-1})$ is a direct summand, so there exists a 
retraction $\rho: Nil_{i-1}(R;\alpha^{-1})\rightarrow H$.  
Let $H_m$ be obtained as above.  We first construct a retraction $\lambda: Nil_{i-1}(R;\alpha^{-1})\rightarrow H_m$. 
Recall that $\mu_{1+mn}$ is an isomorphism on $H_m$, so there exists an integer $l$ so that 
$\mu_{l}\circ\mu_{1+mn}$ is the identity on $H_m$.  We define $\lambda: Nil_{i-1}(R;\alpha^{-1})\rightarrow H_m$ to be 
the composition of the following maps:
$$\xymatrix{
Nil_{i-1}(R;\alpha^{-1})\ar[r]^-{F_m} & Nil_{i-1}(R;\alpha^{-1})\ar[rr]^-{\mu_{2+mn}-G_m} & & Nil_{i-1}(R;\alpha^{-1})\ar[r]^-{\rho}
& H \ar[r]^-{V_m|_{H}} & H_m\ar[r]^-{\mu_{l}} & H_m
}$$
We claim  $\lambda$ is a retraction. 
Note for $x\in H_m$, there exists $y\in H$ with $V_m(y)=x$. We now evaluate
\begin{align*}
\lambda(x) &=(\mu_l\circ V_m\circ\rho\circ(\mu_{2+mn}-G_m)\circ F_m)(x)\\
&=(\mu_l\circ V_m\circ\rho\circ(\mu_{2+mn}-G_m)\circ F_m)(V_m(y))\\
&=(\mu_l\circ V_m\circ\rho\circ(\mu_{2+mn}-G_m)\circ G_m)(y)\\
&=(\mu_l\circ V_m\circ\rho\circ((2+mn)G_m-G_m^2)))(y)\\
&=(\mu_l\circ V_m\circ\rho\circ\mu_{1+mn})(y)\\
&=\big(\mu_l\circ \mu_{1+mn}\big)\big( (V_m\circ\rho)(y)\big)\\
&=\big(\mu_l\circ \mu_{1+mn}\big)\big(V_m(y)\big)\\
&=(\mu_l\circ \mu_{1+mn})(x)\\
&=x
\end{align*}
This verifies $\lambda$ is a retraction. Note also that $\lambda(H)=0$, since $F_m(H)=0$ follows from the fact that
$m \in S$ (recall that integers in $S$ are larger than $r(H)$). 
Hence we are in the situation of Lemma \ref{splitting-lemma}, and we can conclude that $H\oplus H$ also arises as a 
direct summand of $Nil_{i-1}(R;\alpha^{-1})$. Note that, when applying our Lemma \ref{splitting-lemma}, we replaced
the {\it second copy} $H_m$ of $H$ by some other (isomorphic) subgroup, but kept the {\it first copy} of $H$ to be 
the original $H$.
Hence the direct summand $H\oplus H$ does indeed extend the original summand $H$.
Iterating the process, we obtain that $\oplus_{\infty}H$ is a direct 
summand of $Nil_{i-1}(R;\alpha^{-1})$. This completes the proof of {\bf Theorem B} in the case where $i\ge 1$.

\vskip 10pt

\noindent \underline{Case $i\leq 1$.} Next, let us consider the case of the Farrell Nil-groups $NK_i(R, \alpha^{-1})$ 
where $i\leq 1$. For these, the proof of {\bf Theorem B} proceeds via a  (descending) induction on $i$, with the 
case $i=1$ having been established above.

We remind the reader of the standard technique for extending results known for $K_1$ to lower $K$-groups. Take
the ring $\Lambda \mathbb Z$ consisting of all $\mathbb N \times \mathbb N$ 
matrices with entries in $\mathbb Z$ which contain only finitely
many non-zero entries in each row and each column, and quotient out by the ideal $I \triangleleft \Lambda \mathbb Z$
consisting of all matrices
which vanish outside of a finite block. This gives the ring $\Sigma \mathbb Z = \Lambda \mathbb Z \slash I$, and 
we can now define the {\it suspension functor} on the category of rings by tensoring with the ring $\Sigma \mathbb Z$,
i.e. sending a ring $R$ to the ring 
$\Sigma (R):=\Sigma \mathbb Z \otimes R$, and a morphism $f: R\rightarrow S$ to the morphism $Id\otimes f :
\Sigma(R) \rightarrow \Sigma(S)$. The functor $\Sigma$ has the property that there are natural 
isomorphisms $K_{i}(R) \cong K_{i+1}(\Sigma(R))$ (for all $i\in \mathbb Z$). Moreover, there is a natural 
isomorphism $\Sigma \big( R_\alpha[t]\big) \cong (\Sigma R)_{Id\otimes \alpha}[t]$, which induces a 
commutative square
$$
\xymatrix{K_i(R_{\alpha}[t]) \ar[d]^ {\cong} \ar[r] & K_i(R) \ar[d]^{\cong} \\
K_{i+1}\big( (\Sigma \mathbb Z \otimes R)_{Id\otimes \alpha}[t]\big) \ar[r] & K_{i+1}(\Sigma \mathbb Z \otimes R)
}
$$
By induction, for each $m\in \mathbb N$, this allows us to identify $NK_{1-m}(R, \alpha)$ with 
$NK_1(\Sigma^m R, Id^{\otimes m} \otimes \alpha)$, where $\Sigma^m$
denotes the $m$-fold application of the functor $\Sigma$. Obviously, if the automorphism $\alpha$
has finite order in $\text{Aut}(R)$, the induced automorphism $Id^{\otimes m} \otimes \alpha$ will have
finite order in $\text{Aut}\big( (\Sigma \mathbb Z)^{ \otimes m} \otimes R\big)$.
So for the Farrell Nil-groups $NK_i(R, \alpha)$ with $i\leq 0$, the 
result immediately follows from the special case of $NK_1$ considered above. This completes the proof
of {\bf Theorem B}.


\section{A structure theorem and Nils associated to virtually cyclic groups}\label{Nils-VC}

In this section, we discuss some applications and establish {\bf Theorem C} and {\bf Theorem D}.
For a general ring $R$, we know by {\bf Theorem A} that a non-trivial Nil-group is an {\it infinitely generated} abelian 
group. While finitely generated abelian groups have a very nice structural theory, the
picture is much more complicated in the infinitely generated case (the reader can consult \cite[Chapter 4]{Ro} for
an overview of the theory). 
If one restricts to abelian (torsion) groups of {\it finite exponent}, then it is an old result of Pr\"ufer  \cite{pruf}
that any such group is a direct sum of cyclic groups (see \cite[item 4.3.5 on pg. 105]{Ro} for a proof). 


\subsection{Proof of Theorem C} \label{finite}

We can now explain how our {\bf Theorem B} allows us to obtain a structure theorem for certain Nil-groups.  Let $R$ be a countable ring and $\alpha: R\rightarrow R$ be an automorphism
of finite order. 
Then by Proposition \ref{countable} of the Appendix, we know that $NK_i(R,\alpha)$ is a countable group. 
If in addition $NK_i(R,\alpha)$ has finite exponent, then by the result of Pr\"ufer mentioned above, it follows that 
$NK_i(R, \alpha)$
decomposes as a countable direct sum of cyclic groups of prime power order, each of which appears with some
multiplicity.  In view of our {\bf Theorem B}, any summand which occurs
must actually occur infinitely many times. Since the exponent of the Nil-group is finite, 
there is an upper bound on the prime power orders that can appear, and hence there are only finitely many
possible isomorphism types of summands. Let $H$ be the direct sum of a single copy of each 
cyclic group of prime power order which appear as a summand in $NK_i(R, \alpha)$. It follows
immediately that
$\bigoplus_{\infty} H\cong NK_i(R, \alpha)$. This completes the proof of {\bf Theorem C}.



\subsection{Farrell-Jones Isomorphism Conjecture}

In applications to geometric topology, the rings of interest are typically integral group rings $\mathbb ZG$. For
computations of the $K$-theory of such groups, the key tool is provided by the ($K$-theoretic) {\it Farrell-Jones 
Isomorphism Conjecture} \cite{fj}. Davis and L\"uck \cite{DL} gave a general framework for the formulations of various 
isomorphism conjectures. In particular, they constructed for any group $G$, an Or$G$-spectrum, i.e. a functor $\textbf K_{\mathbb Z}: \text{Or}G\rightarrow \textbf{Sp}$, where Or$G$ is the orbit category of $G$ (objects are cosets $G/H, H<G$ and morphisms are $G$-maps) and $\textbf {Sp}$ is the category of spectra. This functor has the property that $\pi_n(\textbf K_{\mathbb Z}(G/H))=K_n(\mathbb ZH)$. As 
an ordinary spectrum can be used to construct a generalized homology theory, this Or$G$-spectrum $\textbf K_{\mathbb Z}$ was used to construct  a $G$-equivariant homology theory $H^G_*(-; \textbf K_{\mathbb Z})$.  It has the property that $H_n^G(G/H; \textbf K_{\mathbb Z})=\pi_n(\textbf K_{\mathbb Z}(G/H))=K_n(\mathbb ZH)$ (for all $H<G$ and $n\in \mathbb Z$). In particular, on a point, $H^G_n(*; \textbf K_{\mathbb Z})=H^G_n(G/G; \textbf K_{\mathbb Z})=K_n(\mathbb ZG)$. Applying this homology theory to any $G$-CW-complex $X$, the 
obvious $G$-map $X\rightarrow *$ gives rise to an {\it assembly map}:
$$H_n^G(X; \textbf K_{\mathbb Z})\rightarrow H_n^G(*; \textbf K_{\mathbb Z})\cong K_n(\mathbb ZG).$$
The Farrell-Jones isomorphism conjecture asserts that, when the space $X$ is a model for the classifying space for
$G$-actions with isotropy in the virtually cyclic subgroups of $G$, then the above assembly map is an isomorphism.
Thus, the conjecture roughly predicts that the $K$-theory of an integral group ring $\mathbb ZG$ is
determined by the $K$-theory of the integral group rings of the virtually cyclic subgroups of $G$, assembled
together in some homological fashion. 

In view of this
conjecture, one can view the $K$-theory of virtually cyclic groups as the ``basic building blocks'' for the $K$-theory
of general groups. Focusing on such a virtually cyclic group $V$, one can consider the portion of the $K$-theory
that comes from the finite subgroups of $V$. This would be the image of the assembly map:
$$H_n^V(\underbar EV; \textbf K_{\mathbb Z})\rightarrow H_n^V(*; \textbf K_{\mathbb Z})\cong K_n(\mathbb ZV)$$
where $\underbar EV$ is a model for the classifying space for proper $V$-actions. While this map is always 
split injective (see \cite{B}), it is not surjective in general. Thus to understand the 
$K$-theory of a virtually cyclic group, we need to understand the $K$-theory of finite groups, and to understand the
cokernels of the above assembly map. The cokernels of the above assembly map can also be interpreted as the 
obstruction to reduce the family of virtually cyclic groups used in the 
Farrell-Jones isomorphism conjecture to the family of finite groups - this is the {\it transitivity principle} (see
\cite[Theorem A.10]{fj}).   
Our {\bf Theorem D} gives some structure for the cokernel of the 
assembly map.


\subsection{Proof of Theorem D} Let $V$ be a virtually cyclic group. Then one has that $V$ is either of the form 
(i) $V=G\rtimes _\alpha \mathbb Z$, where $G$ is a finite group and $\alpha \in \text{Aut}(G)$, or is of the form
(ii) $V= G_1*_HG_2$, where $G_i$, $H$ are finite groups and $H$ is of index two in both $G_i$. 

Let us first consider case (i). In this case, the integral group
ring $\mathbb Z[V]$ is isomorphic to the ring $R_{\hat \alpha}[t,t^{-1}]$, the $\hat \alpha$-twisted ring of 
Laurent polynomials
over the coefficient ring $R=\mathbb Z[G]$, where $\hat \alpha \in \text{Aut}(\mathbb Z[G])$ is the ring automorphism
canonically induced by the group automorphism $\alpha$. Then it is known (see \cite[Lemma 3.1]{DQR}) that the 
cokernel we are interested in consists of the direct sum of the Farrell Nil-group $NK_n(\mathbb ZG, \hat \alpha)$
and the Farrell Nil-group $NK_n(\mathbb ZG, \hat \alpha^{-1})$. Applying {\bf Theorem C}  and Corollary \ref{nk0} 
to these two Farrell Nil-groups, we are done.

In case (ii), we note that $V$ has a canonical surjection onto the infinite dihedral group 
$D_\infty =\mathbb Z_2*\mathbb Z_2$, obtained by surjecting each $G_i$ onto $G_i/H \cong \mathbb Z_2$. Taking
the preimage of the canonical index two subgroup $\mathbb Z \leq D_\infty$, we obtain a canonical index two
subgroup $W\leq V$. The subgroup $W$ is a virtually cyclic group of type (i), and is of the form 
$H\rtimes _\alpha \mathbb Z$, where $\alpha \in \text{Aut}(H)$. Hence it has associated Farrell Nil-groups
$NK_n(\mathbb ZH, \hat \alpha)$. 

The cokernel of the relative assembly map for the group $V$ is a {\it Waldhausen Nil-group} associated to the 
splitting of $V$ (see \cite[Lemma 3.1]{DQR}). It was recently shown that this Waldhausen Nil-group is always 
isomorphic to a single copy of 
the Farrell Nil-group $NK_n(\mathbb ZH, \hat \alpha)$ associated to the canonical index two subgroup $W\leq V$
(see for example \cite{DKR}, \cite{DQR}, or for an earlier result in a similar vein \cite{LO}). 
Again, combining this with our {\bf Theorem C} and Corollary \ref{nk0}, we are done, completing the proof of 
{\bf Theorem D}.


\section{Applications and Concluding Remarks}\label{concluding}

We conclude this short note with some further applications and remarks. 


\subsection{Waldhausen's A-theory}

Recall that Waldhausen \cite{Wa} introduced a notion of algebraic $K$-theory $A(X)$ of a topological 
space $X$.
Once the $K$-theoretic contribution has been split off, one is left with the finitely dominated version of the algebraic
$K$-theory $A^{fd}(X)$. This finitely dominated version satisfies the ``fundamental theorem of algebraic $K$-theory'',
in that one has a homotopy splitting:
\begin{equation}\label{A-theory}
A^{fd}(X\times S^1) \simeq A^{fd}(X) \times BA^{fd}(X) \times NA_+^{fd}(X) \times NA_-^{fd}(X)
\end{equation}
see \cite{HKVWW} (the reader should compare this with the corresponding fundamental theorem of algebraic 
$K$-theory for rings, see \cite{Gr}). The Nil-terms appearing in this splitting have been studied by Grunewald,
Klein, and Macko \cite{GKM}, who defined Frobenius and Verschiebung operations, $F_n, V_n$, on the 
homotopy groups  $\pi_*\big(NA_{\pm}^{fd}(X)\big)$. In particular, they show that the composite $V_n\circ F_n$ 
is multiplication by $n$ \cite[Proposition 5.1]{GKM}, and that for any element $x\in \pi_i\big(NA_{\pm}^{fd}(X)\big)$ 
of finite order, one has $F_n(x)=0$ for all sufficiently large $n$ (see the discussion in 
\cite[pg. 334, Proof of Theorem 1.1]{GKM}). Since these two properties are the only ones used in our proofs,
an argument identical to the proof of {\bf Theorem B} gives the:

\begin{prop}
Let $X$ be an arbitrary space, and let $NA_{\pm}^{fd}(X)$ be the associated Nil-spaces arising in the 
fundamental theorem of algebraic $K$-theory of spaces. Then if $H \leq \pi_i\big(NA_{\pm}^{fd}(X)\big)$ is any
finite subgroup, then
$$\bigoplus _\infty H \leq \pi_i\big(NA_{\pm}^{fd}(X)\big).$$ 
Moreover, if $H$ is a direct summand in $\pi_i\big(NA_{\pm}^{fd}(X)\big)$, then so is $\bigoplus _\infty H$.
\end{prop}

\begin{rem}
An interesting question is whether there exists a ``twisted'' version of the splitting in equation (\ref{A-theory}),
which applies to bundles $X\rightarrow W\rightarrow S^1$ over the circle (or more generally, to approximate 
fibrations over the circle), and provides a homotopy splitting
of the corresponding $A^{fd}(W)$ in terms of spaces attached to $X$ and the holonomy map. 
\end{rem}


\subsection{Cokernels of assembly maps}

For a general group $G$, one would expect from the Farrell-Jones isomorphism Conjectures that the
cokernel of the relative assembly map for $G$ should be ``built up'', in a homological manner, from the cokernels
of the relative assembly maps of the various virtually cyclic subgroups of $G$ (see for example \cite{LO2} for
an instance of this phenomenon). In view of our {\bf Theorem D}, the following question seems relevant:

\vskip 7pt

\noindent {\bf Question:} Can one find a group $G$, an index $i\in \mathbb Z$, and a finite subgroup $H$, with
the property that $H$ embeds in $ \text{CoKer}\Big( h_i^G(\underbar EG) \rightarrow K_i(\mathbb Z[G])\Big)$, but $\bigoplus_\infty H$ does not?

\vskip 6pt

In other words, we are asking whether contributions from the various Nil-groups of the virtually cyclic subgroups
of $G$ could {\it partially cancel out in a cofinite manner}. Note the following special case of this question: is
there an example for which this cokernel is a non-trivial finite group?


\subsection{Exotic Farrell Nil-groups}

Our {\bf Theorem C} establish that,  for a countable \textit{tame} ring,  meaning the associated Farrell Nil-groups 
has finite exponent, the associated Farrell Nil-groups, while infinitely generated, still 
remain reasonably well behaved, i.e. are countable direct sums of a fixed finite group. In contrast, for a general 
ring $R$ (or
even, a general integral group ring $\mathbb ZG$), all we know about the non-trivial Farrell Nil-groups is that they 
are infinitely generated abelian groups. Of course, the possibility of having infinite exponent {\it a priori} allows for 
many strange possibilities, e.g.
the rationals $\mathbb Q$, or the Pr\"ufer $p$-group $\mathbb Z(p^\infty)$ consisting of all complex $p^i$-roots
of unity ($i\geq 0$). We can ask:

\vskip 10pt

\noindent {\bf Question:} Can one find a ring $R$, automorphism $\alpha \in \text{Aut}(R)$, and $i\in \mathbb Z$, 
so that $NK_i(R, \alpha) \cong \mathbb Q$? How about $NK_i(R, \alpha) \cong \mathbb Z(p^\infty)$? What about
if we require the ring to be an integral group ring $R=\mathbb ZG$?

\begin{rem} Grunewald \cite[Theorem 5.10]{gr-topology} proved that  for every group $G$ and every group
automorphism $\alpha$ of finite order,  $NK_i(\mathbb{Q}G, \alpha)$ is a vector space over the rationals after killing torsion elements for all $i\in\mathbb Z$. However this still leaves the possibility that they may vanish. 
\end{rem}

Or rather, in view of our results, the following question also seems natural:

\vskip 10pt

\noindent {\bf Question:} What conditions on the ring $R$, automorphism $\alpha \in \text{Aut}(R)$, and 
$i\in \mathbb Z$, are sufficient to ensure $NK_i(R, \alpha)$ is a torsion group of finite exponent?  Does $NK_i(\mathbb Z G; \alpha)$ have finite exponent for all polycyclic-by-finite
groups when $\alpha$ is of finite order?

\vskip 10pt

Finally, while this paper completes our understanding of the finiteness properties of Farrell Nil-groups 
associated with {\it finite order} ring automorphisms, nothing seems to be known about the Nil-groups
associated with {\it infinite order} ring automorphisms. This seems like an obvious direction for further research.

\section{Appendix}
In this appendix, we give a short discussion on the cardinality of Nil-groups. The following proposition is needed 
in the proof of our {\bf Theorem C} -- while presumably well-known to experts, we were unable to find it in the
literature.

\begin{prop}\label{countable} Let $R$ be a countable ring and $\alpha: R\rightarrow R$ be a ring automorphism.  Then the groups $K_i(R)$ and $NK_i(R;\alpha )$ are countable for all $i\in\mathbb Z$. 
\end{prop}

\begin{proof} Since $NK_i(R;\alpha)$ is a subgroup of $K_i(R_{\alpha}[t])$ and $R_{\alpha}[t]$ is countable when 
$R$ is countable, it is enough to show $K_i(R)$ is countable when $R$ is countable. So let us focus on $K_i(R)$. 

\vskip 5pt

We first use Quillen's +-construction to treat the case where $i\ge 1$.  Consider the infinite general linear group 
$GL(R)$. Being the countable union of countable groups $GL_n(R)$ ($n\in\mathbb N$), we see that $GL(R)$
is countable. Hence the simplicial complex
spanned by the group elements of $GL(R)$, which is a model for the universal space for free $GL(R)$-actions, is
also countable. Then the quotient $BGL(R)$ is of course a countable $CW$-complex. Performing Quillen's 
$+$-construction to $BGL(R)$, we obtain the algebraic $K$-theory space $BGL(R)^+$ with 
$K_i(R):=\pi_i(BGL(R)^+)$, for $i\ge 1$. Note that $BGL(R)^+$ is obtained from $BGL(R)$ by attaching 2-cells 
and 3-cells indexed by some generating set of the commutator subgroup of $GL(R)$, hence $BGL(R)^{+}$ is 
a countable $CW$-complex.  More details of Quillen's +-construction can be found, for example, in 
\cite[Theorem 5.2.2]{Ros}.  

We now show the homotopy groups of a countable $CW$-complex is countable.
By filtering a countable $CW$-complex by its countably many finite subcomplexes, it suffices to show homotopy 
groups of finite $CW$-complexes are countable. So let us assume $X$ is a finite $CW$-complex. Since every finite 
$CW$-complex has the homotopy type of a finite simplicial complex, we may assume $X$  is a finite simplicial 
complex. Fix  a triangulation $\Delta_i$ of $S^i$. The set of all iterated barycentric subdivisions of $\Delta_i$
is countable. Fix a vertex in $\Delta_i$ and a vertex in $X$ as base points. By simplicial approximation, any element 
in $\pi_i(X)$ can be represented by a simplicial map from some iterated barycentric subdivision of $\Delta_i$ to $X$. 
But the set of such simplicial maps is clearly countable, hence $\pi_i(X)$ is countable. Thus $K_i(R)$ is countable 
when $i\ge 1$.

Now let us consider the case when $i<1$. First, we consider $i=0$. Let Idem$(R)$ be the set of idempotent matrices 
in $M(R)$, where $M(R)$ is the union of all $n\times n$ matrices over $R$, ($n\in\mathbb N$).  $GL(R)$ acts on
Idem$(R)$ by conjugation, denote the quotient by Idem$(R)/GL(R)$. This is a semigroup and $K_0(R)$ can be 
identified with the Grothendieck group associated to this semigroup (see \cite[Theorem 1.2.3]{Ros}). Therefore 
$K_0(R)$ is countable since Idem$(R)$
is countable. Now when $i<0$, the negative $K$-groups $K_i(R)$ can be inductively defined to be the cokernel 
of the natural map (see \cite [Definition 3.3.1]{Ros})
$$K_{i+1}(R[t])\oplus K_{i+1}(R[t^{-1}])\rightarrow K_{i+1}R[t,t^{-1}]$$
Note when $R$ is countable, $R[t], R[t^{-1}]$ and $R[t,t^{-1}]$ are all countable. Hence their $K_0$-groups are all 
countable. Thus we inductively have $K_i(R)$ are  countable for all $i<0$.  This completes the proof. 
\end{proof}

\end{document}